\documentclass{amsart}
\usepackage[T1]{fontenc}
\usepackage{txfonts}
\usepackage{geometry}
\usepackage[all]{xy}
\usepackage{graphicx}
\usepackage{enumerate}
\usepackage{amscd}
\usepackage{mathrsfs}
\usepackage{amsmath}
\usepackage{amssymb}
\usepackage[hidelinks]{hyperref}
\usepackage{fixltx2e}

\newtheorem{lemma}{Lemma}
\newtheorem{proposition}{Proposition}
\newtheorem{theorem}{Theorem}
\newtheorem{corollary}{Corollary}

\theoremstyle{remark}
\newtheorem*{remark}{Remark}

\DeclareMathOperator{\Spec}{Spec}

\DeclareMathOperator{\Pic}{Pic}

\DeclareMathOperator{\Gal}{Gal}
\DeclareMathOperator{\cris}{cris}
\DeclareMathOperator{\Br}{Br}

\DeclareMathOperator{\dR}{dR}
\DeclareMathOperator{\cha}{char}

\DeclareMathOperator{\spe}{sp}

\DeclareMathOperator{\red}{red}

\DeclareMathOperator{\Hom}{Hom}

\DeclareMathOperator{\cont}{cont}

\DeclareMathOperator{\im}{im}
\DeclareMathOperator{\pr}{pr}

\DeclareMathOperator{\ur}{ur}

\DeclareMathOperator{\new}{new}
\DeclareMathOperator{\old}{old}
\DeclareMathOperator{\et}{\acute{e}t}
\DeclareMathOperator{\tee}{t}

\newcommand{\trans}[1]{\,^{\tee}{#1}}

\begin{document}

\title{On Tate's conjecture for elliptic modular surfaces over finite fields}

\author{R\'emi Lodh}
\email{remi.shankar@gmail.com}

\begin{abstract}
For $N\geq 3$, we show Tate's conjecture for the elliptic modular surface $E(N)$ of level $N$ over $\mathbb{F}_p$ for a prime $p$ satisfying $p\equiv 1\mod N$ outside of a set of primes of density zero. We also prove a strong form of Tate's conjecture for $E(N)$ over any finite field of characteristic $p$ prime to $N$ under the assumption that the formal Brauer group of $E(N)$ is of finite height.
\end{abstract}

\maketitle

\section*{Introduction}
In this paper we study cohomology classes of divisors on the elliptic modular surface $E(N)$ of level $N$, where $N\geq 3$.  By definition, $E(N)$ is the universal object over the moduli space $X(N)$ of generalized elliptic curves with level $N$ structure. Fix a prime $p$ which does not divide $N$. Our first result is the following theorem, which goes back to Shioda \cite[Appendix]{shioda1} for $N=4$.

\begin{theorem}[Corollary \ref{tate}]\label{main}
Assume the partial semi-simplicity conjecture (\ref{SS}) is true for $E(N)_{\mathbb{F}_p}$. If $p\equiv 1\mod N$, then Tate's conjecture holds for each connected component of $E(N)_{\mathbb{F}_p}$. Moreover, the Mordell-Weil group of a generic fibre of $E(N)_{\mathbb{F}_p}\to X(N)_{\mathbb{F}_p}$ is isomorphic to $(\mathbb{Z}/N)^2$.
\end{theorem}

The partial semi-simplicity conjecture for a smooth projective surface $S$ over a finite field $k$ is the following statement: if $q=\#k$, then
\begin{equation*}\label{SS}\tag{PS}
H^2_{\cris}(S)^{(F_q-q^2)=0}=H^2_{\cris}(S)^{F_q=q}
\end{equation*}
where $H^2_{\cris}(S)=H^2_{\cris}(S/W(k))\otimes\mathbb{Q}$ is the second crystalline cohomology group of $S$ and $F_q$ is the $q$-power Frobenius endomorphism. It is a consequence of Tate's conjecture. For $E(N)_{\mathbb{F}_p}$, it is related to estimates for the absolute value of the Fourier coefficients of normalized newforms of weight 3, see \cite{colemanedixhoven}. Such estimates can be shown to hold if the field of coefficients of the newform is totally real, or for CM forms, but the general case is unknown. In any case, we can show that it holds for all primes $p$ outside of a set of density zero, viz.

\begin{corollary}[Corollary \ref{tatecor}]
The conclusion of Theorem \ref{main} holds for all $p\equiv 1\mod N$ outside of a set of primes of density zero.
\end{corollary}

For arbitrary $p$, we can also show Tate's conjecture under a completely different assumption, namely that the formal Brauer group of $E(N)_{\mathbb{F}_p}$ is a formal Lie group of finite height. In fact, writing $T_p\,A=\varprojlim_nA[p^n]$ where $A[m]\subset A$ is the subgroup of $m$-torsion ($m\in\mathbb{Z}$) for any abelian group $A$, we have

\begin{theorem}[Corollary \ref{fhcor}]\label{fh}
If the formal Brauer group of a connected component of $E(N)_{\bar{\mathbb{F}}_p}$ has finite height, then $T_p\Br(E(N)_{\bar{\mathbb{F}}_p})=0$.
\end{theorem}

However, it seems to be a difficult arithmetic problem to determine for which $p$ the reduction of $E(N)$ has finite height formal Brauer group. Shioda \cite{shioda1, shioda2} has shown that $E(4)_{\bar{\mathbb{F}}_p}$ is a K3 surface whose formal Brauer group is of finite height if and only if $p\equiv 1\mod 4$, but we do not know what happens for $N>4$.

The starting point for the proofs of both Theorems \ref{main} and \ref{fh} is the following fact:
\begin{equation*}\label{hodgetateK}\tag{HT}
V_p\Br(E(N)_{\bar{\mathbb{Q}}})\;\text{is a Hodge-Tate representation with weights}\pm 1.
\end{equation*}
Here $\Br(-):=H^2(-,\mathbb{G}_m)$ denotes the cohomological Brauer group and for any abelian group $A$ we write $V_pA:=T_pA\otimes\mathbb{Q}$.

We can summarize our method of proof of Theorem \ref{main} as follows. Let $I_p$ be the extension to $E(N)$ of the morphism of $X(N)$ obtained by multiplying the level structure by $p\in(\mathbb{Z}/N)^*$ and let $U\subset V_p\Br(E(N)_{\bar{\mathbb{Q}}})$ be the subset on which $I_p$ acts trivially. Then (modulo (\ref{SS})) (\ref{hodgetateK}) and the Eichler-Shimura congruence relation imply $D_{\cris}(U)^{\varphi=1}=0$, where $\varphi$ is the Frobenius. For $p\equiv 1\mod N$, $I_p$ is the identity and the theorem follows. In the case $p\not\equiv 1\mod N$ we only know Shioda's result \cite{shioda2} for $N=4$.

The proof of Theorem \ref{fh} is much more straightforward: it makes no use of Hecke operators and applies more generally to any proper smooth surface over a finite extension of $\mathbb{Z}_p$ satisfying (\ref{hodgetateK}) and whose formal Brauer group has finite height (cf. Corollary \ref{subquot}), thereby establishing a relationship between (\ref{hodgetateK}) and Tate's conjecture in this case.

\subsection*{Notation}
All sheaves and cohomology will be for the \'etale topology, except where otherwise stated.

We denote by $k$ a perfect field of characteristic $p>0$, $W=W(k)$ its ring of Witt vectors, $K_0=W[1/p]$, $\bar{k}$ an algebraic closure of $k$, $\bar{K}$ be an algebraic closure of $K_0$, $G_{K_0}=\Gal(\bar{K}/K_0)$, $\hat{\bar{K}}$ the completion of $\bar{K}$ for the $p$-adic norm.

\section{A general result}

We assume familiarity with Fontaine's basic theory \cite{font, font2, colfont}.

\subsection{Autodual crystalline representations}
Let $V$ be a $p$-adic representation of $G_{K_0}$. We say that $V$ is \emph{autodual} if it is isomorphic to its dual, i.e. it has a non-degenerate bilinear form
\[ V\otimes_{\mathbb{Q}_p}V\to\mathbb{Q}_p \]
which is a homomorphism of $G_{K_0}$-modules.

\begin{proposition}\label{key}
Let $V$ be an autodual crystalline representation of $G_{K_0}$ and let $D:=D_{\cris}(V)$ be the associated filtered $\varphi$-module. Suppose the endomorphism $T:=\varphi+\varphi^{-1}$ of $D$ satisfies $T(F^1D)\subset F^1D$. If $D^{(\varphi-1)^2=0}=D^{\varphi=1}$ and $V^{G_{K_0}}=0$, then $D^{\varphi=1}=0$.
\end{proposition}
\begin{proof}
The bilinear form on $V$ induces a non-degenerate bilinear form $\cdot$ on $D$. Endow $D^{\varphi=1}$ with the induced filtration from $D$. Since $V$ is crystalline we have $F^0D^{\varphi=1}=V^{G_{K_0}}=0$. If $D^{\varphi=1}=F^0D^{\varphi=1}$, then we are done. If not, then there is $i<0$ and $x\in F^iD^{\varphi=1}\setminus F^{i+1}D^{\varphi=1}$. Since $V$ is autodual, the map $c:D\to D^*:=\Hom_{K_0}(D,K_0)$ induced by $\cdot$ is an isomorphism of filtered $\varphi$-modules, so we have $x^*:=c(x)\in F^iD^*\setminus F^{i+1}D^*$. Note that $x^*$ is the map $D\ni y\mapsto x\cdot y\in K_0$. Since by definition
\[ F^iD^*=\{f\in D^*:f(F^jD)\subset F^{j+i}K_0\;\forall j\in\mathbb{Z}\} \]
the condition $x^*\notin F^{i+1}D^*$ means that there is $j$ such that $x^*(F^jD)\not\subset F^{j+i+1}K_0$, where $K_0$ has the trivial filtration, i.e.
\[ F^kK_0=
\begin{cases}
K_0 & k\leq 0 \\
0 & k>0.
\end{cases} \]
If $x^*(F^jD)\not\subset F^{j+i+1}K_0$, then we must have $x^*(F^jD)\neq 0$, hence $x^*(F^jD)=K_0$. So to say that $x^*(F^jD)\not\subset F^{j+i+1}K_0$ but $x^*(F^jD)\subset F^{i+j}K_0$ is equivalent to the condition $j+i=0$. Hence $j=-i>0$, and there is an element $y\in F^1D$ such that $x\cdot y\neq 0$.

Now, up to dividing $y$ by $x\cdot y$ we may assume that $x\cdot y\in\mathbb{Q}_p$. Let $P(t)\in W[t]$ be such that $P(T)y=0$. Since $\varphi(x)=x$ we have $x\cdot T(d)=T(x\cdot d)$ for all $d\in D$, hence
\[ 0=x\cdot P(T)y=P(T)(x\cdot y)=(x\cdot y)P(2). \]
So $P(2)=0$ and we deduce that $P(t)=(t-2)^eQ(t)$ for some $e\in\mathbb{N}$ and some polynomial $Q(t)$ not divisible by $t-2$. Let $z:=Q(T)y$. Note that $x\cdot z=(x\cdot y)Q(2)\neq 0$. Multiplying the equation $(T-2)^ez=0$ by $\varphi^e$ we find $(\varphi-1)^{2e}z=0$, hence $\varphi(z)=z$ since $D^{(\varphi-1)^2=0}=D^{\varphi=1}$. As $F^1D$ is stable under $T$ by assumption, we have $z\in F^1D$. Thus, $z\in F^1D^{\varphi=1}\subset V^{G_{K_0}}=0$, hence $x=0$, and this completes the proof.
\end{proof}

\begin{remark}
The above argument no longer works if one replaces $\varphi$ by a power $\varphi^r$. The problem is related to the fact that, unlike the case $r=1$, for $r>1$ we may have $F^1B_{\cris}^{\varphi^r=1}\neq 0$.
\end{remark}

\subsection{Application to surfaces}\label{mainsection}
Let $E\to\Spec(W)$ be a smooth projective morphism with geometrically connected fibres of dimension 2. Let $K_0^{\ur}$ be the maximal unramified extension of $K_0$ in $\bar{K}$. The Kummer sequence gives an exact sequence of $G_{K_0}$-representations
\[ 0\to NS(E_{\bar{K}})\otimes\mathbb{Q}_p\to H^2_{\et}(E_{\bar{K}},\mathbb{Q}_p)(1)\to V_p\Br(E_{\bar{K}})\to 0 \]
where $NS:=\Pic/\Pic^0$ is the N\'eron-Severi group. By $p$-adic Hodge theory, applying the functor $D_{\cris}$ we get an exact sequence
\[ 0\to D_{\cris}(NS(E_{\bar{K}})_{\mathbb{Q}_p})\to H^2_{\cris}(E_k/K_0)[1]\to D_{\cris}(V_p\Br(E_{\bar{K}}))\to 0 \]
where for a filtered $\varphi$-module $D$ we denote $D[1]$ the filtered $\varphi$-module whose underlying $K_0$-module is $D$ with $\varphi_{D[1]}:=p^{-1}\varphi_D$ and $F^iD[1]:=F^{i+1}D$. On the other hand, there is the specialization map
\[ \spe:NS(E_{\bar{K}})\to NS(E_{\bar{k}}) \]
(\cite[Exp. E, appendix, 7.12]{sga6}) which is functorial in $E$ (hence $G_{K_0}$-equivariant) and injective up to torsion (since the N\'eron-Severi group is finitely generated this can be checked using $l$-adic cohomology). So $NS(E_{\bar{K}})\otimes\mathbb{Q}_p$ is an unramified discrete representation of $G_{K_0}$ hence is $K_0^{\ur}$-admissible. Thus, $D_{\cris}(NS(E_{\bar{K}})_{\mathbb{Q}_p})=\left(NS(E_{\bar{K}})\otimes K_0^{\ur}\right)^{\Gal(K_0^{\ur}/K_0)}$, and $D_{\cris}(NS(E_{\bar{k}})_{\mathbb{Q}_p})=\left(NS(E_{\bar{k}})\otimes K_0^{\ur}\right)^{\Gal(K_0^{\ur}/K_0)}$ and we have a commutative diagram
\[ \xymatrix{
D_{\cris}(NS(E_{\bar{K}})_{\mathbb{Q}_p}) \ar[r]^{\spe} \ar[d] & D_{\cris}(NS(E_{\bar{k}})_{\mathbb{Q}_p}) \ar[d]^{c_1} \\
H^2_{\cris}(E_k/K_0)[1] \ar@{=}[r] & H^2_{\cris}(E_k/K_0)[1]
} \]
where $c_1$ is the first Chern class. In fact, $c_1$ is injective since
\[ NS(E_{\bar{k}})_{\mathbb{Q}_p}\subset \left(H^2_{\cris}(E_k/K_0)[1]\otimes_{K_0}K^{\ur}_0\right)^{\varphi=1}. \]
Therefore, defining $C:=H^2_{\cris}(E_k/K_0)[1]/D_{\cris}(NS(E_{\bar{k}})_{\mathbb{Q}_p})$, we have a commutative diagram with exact rows
\[ \xymatrix{
0 \ar[r] & D_{\cris}(NS(E_{\bar{K}})_{\mathbb{Q}_p}) \ar[r] \ar[d] & H^2_{\cris}(E_k/K_0)[1] \ar[r] \ar@{=}[d] & D_{\cris}(V_p\Br(E_{\bar{K}})) \ar[r] \ar[d] & 0 \\
0 \ar[r] & D_{\cris}(NS(E_{\bar{k}})_{\mathbb{Q}_p}) \ar[r] & H^2_{\cris}(E_k/K_0)[1]\ \ar[r] & C \ar[r] & 0
} \]
and so setting $M:=D_{\cris}(NS(E_{\bar{k}})_{\mathbb{Q}_p})/D_{\cris}(NS(E_{\bar{K}})_{\mathbb{Q}_p})$ we deduce an exact sequence of $\varphi$-modules
\[ 0\to M\to D_{\cris}(V_p\Br(E_{\bar{K}}))\to C\to 0. \]

\begin{theorem}\label{mainth}
Let $D:=D_{\cris}(V_p\Br(E_{\bar{K}}))$ and $T:=\varphi+\varphi^{-1}$. If $D^{(\varphi-1)^2=0}=D^{\varphi=1}$, $T(F^1D)\subset F^1D$ and $V_p\Br(E_{\bar{K}})^{G_{K_0}}=0$, then $M^{\varphi=1}=0=C^{\varphi=1}$.
\end{theorem}
\begin{proof}
By Poincar\'e duality, cup product is non-degenerate on $H^2_{\et}(E_{\bar{K}},\mathbb{Q}_p)(1)$, and since numerical and algebraic equivalence coincide up to torsion for divisors (\cite[9.6.17]{kleiman}), it is also non-degenerate on $NS(E_{\bar{K}})_{\mathbb{Q}_p}$. It follows that $V_p\Br(E_{\bar{K}})\cong\left(NS(E_{\bar{K}})_{\mathbb{Q}_p}\right)^{\bot}$ has a non-degenerate symmetric bilinear form we may apply Proposition \ref{key} to obtain $D^{\varphi=1}=0$. Moreover, the restriction of this form to $M$ is non-degenerate since cup product is non-degenerate on both $D_{\cris}(NS(E_{\bar{K}})_{\mathbb{Q}_p})$ and $D_{\cris}(NS(E_{\bar{k}})_{\mathbb{Q}_p})$. Thus, $C\cong M^{\bot}$ and hence $C^{\varphi=1}=0$.
\end{proof}

\begin{corollary}\label{maincor}
Under the assumptions of Theorem \ref{mainth}, Tate's conjecture holds for $E_k$ and we have
\[ NS(E_{\bar{K}})^{G_{K_0}}\otimes\mathbb{Q}=NS(E_{\bar{k}})^{G_{K_0}}\otimes\mathbb{Q}. \]
\end{corollary}
\begin{proof}
Tate's conjecture is well-known to be equivalent to the statement $C^{\varphi=1}=0$ (\cite{milneAT}). For the second statement, note that we have an exact sequence
\[ 0\to D_{\cris}(NS(E_{\bar{K}})_{\mathbb{Q}_p})^{\varphi=1}\to D_{\cris}(NS(E_{\bar{k}})_{\mathbb{Q}_p})^{\varphi=1}\to M^{\varphi=1} \]
so since $M^{\varphi=1}=0$ we have
\[ \left(NS(E_{\bar{K}})\otimes\mathbb{Q}_p\right)^{G_{K_0}}=D_{\cris}(NS(E_{\bar{K}})_{\mathbb{Q}_p})^{\varphi=1}=D_{\cris}(NS(E_{\bar{k}})_{\mathbb{Q}_p})^{\varphi=1}=\left(NS(E_{\bar{k}})\otimes\mathbb{Q}_p\right)^{G_{K_0}} \]
as claimed.
\end{proof}

\section{Elliptic modular surfaces}
We fix throughout an integer $N$ and a prime number $p$ which does not divide $N$. 
\subsection{Definition} For $N\geq 3$, let $Y(N)$ to be moduli $\mathbb{Z}[1/N]$-scheme of elliptic curves with level $N$ structure and let $X(N)$ be its modular compatification. $X(N)$ classifies generalized elliptic curves with level $N$ structure whose singular fibres are N\'eron $N$-gons. $X(N)$ is smooth over $\mathbb{Z}[1/N]$ and the normalization of $\mathbb{Z}[1/N]$ in $X(N)$ is $\mathbb{Z}[\zeta_N,1/N]$, where $\zeta_N$ is a primitive $N$th root of unity. See \cite{delrap} for details. We denote the universal generalized elliptic curve by
\[ g:E(N)\to X(N). \]
$E(N)$ is the \emph{elliptic modular surface of level $N$} studied in \cite{shioda1}. That it is smooth over $\mathbb{Z}[1/N]$ follows from the results of \cite[VII]{delrap}.

\subsection{Application of Hodge theory}
Assume $\zeta_N\in W$ (note that this is always possible if $p\equiv 1\mod N$, for then $\zeta_N^p=\zeta_N$, so $\zeta_N\in\mathbb{Z}_p$). To simplify the notation write
\[ E:=E(N)\otimes_{\mathbb{Z}[\zeta_N]}W,\,X:=X(N)\otimes_{\mathbb{Z}[\zeta_N]}W,\,Y:=Y(N)\otimes_{\mathbb{Z}[\zeta_N]}W,\,\Sigma:=X\setminus Y. \]
Let $L$ be the conormal sheaf of the zero section of $g:E\to X$, and let $\omega=\Omega^1_{X}(\log\Sigma)$ denote the line bundle of differential forms on $X$ with logarithmic poles along $\Sigma$.

\begin{theorem}[Faltings]\label{faltings}
There are $G_{K_0}$-equivariant isomorphisms
\[ H^1(Y_{\bar{K}},R^1g_*\mathbb{Q}_p(1))\otimes_{\mathbb{Q}_p}\hat{\bar{K}}=H^1(X,L^{\otimes -1})\otimes_{W}\hat{\bar{K}}(1)\oplus H^0(X,L\otimes\omega)\otimes_{W}\hat{\bar{K}}(-1) \]
\[ \tilde{H}^1(Y_{\bar{K}},R^1g_*\mathbb{Q}_p(1))\otimes_{\mathbb{Q}_p}\hat{\bar{K}}=H^1(X,L^{\otimes -1})\otimes_{W}\hat{\bar{K}}(1)\oplus H^0(X,L\otimes\Omega^1_X)\otimes_{W}\hat{\bar{K}}(-1) \]
where $\tilde{H}^1:=\im(H^1_c\to H^1)$ is the parabolic cohomology.
\end{theorem}
\begin{proof}
This is a special case of the $p$-adic Eichler-Shimura isomorphism of \cite{famodular}.
\end{proof}

Let $I\subset G_{K_0}$ be the inertia group.

\begin{corollary}
If $X_H$ is representable, then $H^1(Y_{\bar{K}},R^1g_*\mathbb{Q}_p(1))$ is a Hodge-Tate representation with weights $\pm 1$. In particular, $H^1(Y_{\bar{K}},R^1g_*\mathbb{Q}_p(1))^{I}=0$.
\end{corollary}

\begin{corollary}\label{ker}
Let $E'=E\times_XY$. Then
\begin{enumerate}[(i)]
\item $H^2(E'_{\bar{K}},\mathbb{Q}_p(1))=H^1(Y_{\bar{K}},R^1g_*\mathbb{Q}_p(1))\oplus\mathbb{Q}_pe$, where $e$ denotes the characteristic class of the zero section of $g$
\item $H^2(E_{\bar{K}},\mathbb{Q}_p(1))^{I}$ is generated as a $\mathbb{Q}_p$-vector space by the characteristic classes of the irreducible components of singular fibres of $g$ together with $e$.
\end{enumerate}
\end{corollary}
\begin{proof}
Since $Y_{\bar{K}}$ is an affine curve, the Leray spectral sequence
\[ H^i(Y_{\bar{K}},R^jg_*\mathbb{Q}_p(1))\Rightarrow H^{i+j}(E'_{\bar{K}},\mathbb{Q}_p(1)) \]
gives an exact sequence
\[ 0\to H^1(Y_{\bar{K}},R^1g_*\mathbb{Q}_p(1))\to H^2(E'_{\bar{K}},\mathbb{Q}_p(1))\to H^0(Y_{\bar{K}},R^2g_*\mathbb{Q}_p(1))\to 0 \]
so $H^2(E'_{\bar{K}},\mathbb{Q}_p(1))^{I}\subset H^0(Y_{\bar{K}},R^2g_*\mathbb{Q}_p(1))=\mathbb{Q}_p$. In fact we must have equality since the class $e$ of the zero section of $g$ cannot be trivial. So $e$ gives a splitting of the sequence, proving (i). For (ii) it suffices to note that the kernel of the map $H^2(E_{\bar{K}},\mathbb{Q}_p(1))\to H^2(E'_{\bar{K}},\mathbb{Q}_p(1))$ is generated by the classes of the components of the fibres over the cusps.
\end{proof}

Note that combined with the Shioda-Tate formula (\cite[1.5]{shioda1}) this implies that the rank of the Mordell-Weil group of the generic fibre of $g$ is zero, a result of Shioda \cite[5.1]{shioda1}.

\begin{corollary}\label{HT}
$V_p\Br(E_{\bar{K}})\otimes_{\mathbb{Q}_p}\hat{\bar{K}}=H^2(E,\mathcal{O}_E)\otimes\hat{\bar{K}}(1)\oplus H^0(E,\Omega^2_E)\otimes\hat{\bar{K}}(-1)$. In particular, $V_p\Br(E_{\bar{K}})^{I}=0$.
\end{corollary}
\begin{proof}
We have $V_p\Br(E_{\bar{K}})\subset V_p\Br(E'_{\bar{K}})$ (\cite[II, 1.10]{brauer}) and the latter is a quotient of $H^1(Y_{\bar{K}},R^1g_*\mathbb{Q}_p(1))$ by the last corollary, hence $V_p\Br(E_{\bar{K}})$ is a Hodge-Tate representation with weights contained in $\{\pm 1\}$. In particular, $V_p\Br(E_{\bar{K}})\otimes_{\mathbb{Q}_p}\hat{\bar{K}}\cap H^1(E,\Omega^1_E)\otimes\hat{\bar{K}}=0$, and so $V_p\Br(E_{\bar{K}})\otimes_{\mathbb{Q}_p}\hat{\bar{K}}\subset H^2(E,\mathcal{O}_E)\otimes\hat{\bar{K}}(1)\oplus H^0(E,\Omega^2_E)\otimes\hat{\bar{K}}(-1)$. Since
\begin{eqnarray*}
\dim_{\mathbb{Q}_p}V_p\Br(E_{\bar{K}}) &=& \dim_{\mathbb{Q}_p} H^2(E_{\bar{K}},\mathbb{Q}_p(1))-\dim_{\mathbb{Q}_p} NS(X_{\bar{K}})\otimes\mathbb{Q}_p \\
&\geq &\dim_{\mathbb{Q}_p} H^2(E_{\bar{K}},\mathbb{Q}_p(1))-\dim_{\hat{\bar{K}}}H^1(E,\Omega^1_E)\otimes\hat{\bar{K}} \\
&=&\dim_{\hat{\bar{K}}}H^2(E,\mathcal{O}_E)\otimes\hat{\bar{K}}(1)+\dim_{\hat{\bar{K}}}H^0(E,\Omega^2_E)\otimes\hat{\bar{K}}(-1)
\end{eqnarray*}
this implies the result.
\end{proof}

\begin{corollary}\label{parabolic}
There is a canonical isomorphism $V_p\Br(E_{\bar{K}})=\tilde{H}^1(Y_{\bar{K}},R^1g_*\mathbb{Q}_p(1))$.
\end{corollary}
\begin{proof}
Let $E':=E\times_XY$ and write $V:=\tilde{H}^1(Y(\mathbb{C}),R^1g_*\mathbb{Z}(1))$. By the classical Eichler-Shimura isomorphism (cf. Theorem \ref{faltings}), $V$ is a weight 0 Hodge structure of type $\{(1,-1),(-1,1)\}$. We have $V\subset H^1(Y(\mathbb{C}),R^1g_*\mathbb{Z}(1))\subset H^2(E'(\mathbb{C}),\mathbb{Z}(1))$ and since $NS(E'_{\mathbb{C}}):=\im(NS(E_{\mathbb{C}})\to H^2(E'(\mathbb{C}),\mathbb{Z}(1))$ is a Hodge structure of type $(0,0)$ we have $(V\cap NS(E'_{\mathbb{C}}))\otimes\mathbb{Q}=0$, hence $V\otimes\mathbb{Q}\subset H^2(E'(\mathbb{C}),\mathbb{Q}(1))/NS(E_{\mathbb{C}}')\otimes\mathbb{Q}$.

Now $E\setminus E'=\coprod_{x\in\Sigma}g^{-1}(x)$ is the disjoint union of the singular fibres of $g$, and we have an exact sequence in singular cohomology
\[ H^2(E(\mathbb{C}),\mathbb{Z}(1))\to H^2(E'(\mathbb{C}),\mathbb{Z}(1))\to \oplus_{x\in\Sigma(\bar{K})}H^3_{g^{-1}(x)}(E(\mathbb{C}),\mathbb{Z}(1)). \]
We also have a commutative diagram
\[ \begin{CD}
0 @>>> NS(E_{\mathbb{C}}) @>>> H^2(E(\mathbb{C}),\mathbb{Z}(1)) @>>> H^2(E(\mathbb{C}),\mathbb{Z}(1))/NS(E_{\mathbb{C}}) @>>> 0 \\
@. @VVV @VVV @VVV @. \\
0 @>>> NS(E_{\mathbb{C}}') @>>> H^2(E'(\mathbb{C}),\mathbb{Z}(1)) @>>> H^2(E'(\mathbb{C}),\mathbb{Z}(1))/NS(E_{\mathbb{C}}') @>>> 0
\end{CD} \]
in which the left vertical map is surjective, hence the right vertical map is injective and we deduce an exact sequence
\[ 0\to H^2(E(\mathbb{C}),\mathbb{Z}(1))/NS(E_{\mathbb{C}})\to H^2(E'(\mathbb{C}),\mathbb{Z}(1))/NS(E_{\mathbb{C}}')\to \oplus_{x\in\Sigma(\bar{K})}H^3_{g^{-1}(x)}(E(\mathbb{C}),\mathbb{Z}(1)). \]
By Poincar\'e duality $H^3_{g^{-1}(x)}(E(\mathbb{C}),\mathbb{Q}(1))^*=H^1(g^{-1}(x)(\mathbb{C}),\mathbb{Q}(1))=\mathbb{Q}(1)$ (since $g^{-1}(x)$ is a N\'eron polygon), hence $\oplus_{x\in\Sigma(\bar{K})}H^3_{g^{-1}(x)}(E(\mathbb{C}),\mathbb{Z}(1))$ is a Hodge structure of weight 2 and therefore the map
\[ V\to \oplus_{x\in\Sigma(\bar{K})}H^3_{g^{-1}(x)}(E(\mathbb{C}),\mathbb{Q}(1)) \]
is zero. Thus,
\[ V\otimes\mathbb{Q}\subset H^2(E(\mathbb{C}),\mathbb{Q}(1))/NS(E_{\mathbb{C}})\otimes\mathbb{Q}. \]
Now by the Eichler-Shimura isomorphism we have $\dim V\otimes\mathbb{Q}=2\dim H^0(X,L\otimes\Omega^1_X)$, and since $H^0(X,L\otimes\Omega^1_X)=H^0(E,\Omega^2_E)$ (\cite[th. 6.8]{schuettshioda}) from Corollary \ref{HT} (and Serre duality) we get $\dim V=\dim V_p\Br(E_{\bar{K}})$. Since $\left(H^2(E(\mathbb{C}),\mathbb{Z}(1))/NS(E_{\mathbb{C}})\right)\otimes\mathbb{Q}_p=V_p\Br(E_{\bar{K}})$ we have $V\otimes\mathbb{Q}_p=V_p\Br(E_{\bar{K}})$.
\end{proof} 

\begin{remark}
Shioda \cite{shioda1} shows that $H^1(E,\Omega^1_E)\otimes\hat{\bar{K}}$ is generated by the classes of divisors, which together with the Hodge-Tate decomposition gives another proof of Corollary \ref{HT}. Combining this with Corollary \ref{parabolic}, this gives another proof that $\tilde{H}^1(Y_{\bar{K}},R^1g_*\mathbb{Q}_p(1))$ is a Hodge-Tate representation with weights $\pm 1$.
\end{remark}

\subsection{Application of Hecke operators}
The Eichler-Shimura congruence relation relates the $p$th Hecke operator $T_p$ to the Frobenius morphism at $p$. We exploit this relationship to obtain the following

\begin{theorem}\label{eichlershimura}
If $p\equiv 1\mod N$ and $K_0=\mathbb{Q}_p$, then $T:=\varphi+\varphi^{-1}$ is an endomorphism of $D:=D_{\cris}(V_p\Br(E_{\bar{K}}))$ which satisfies $T(F^1D)\subset F^1D$.
\end{theorem}
\begin{proof}
Recall (\cite[V, 1.14]{delrap}) that there is a regular proper $\mathbb{Z}[1/N]$-scheme $X(N,p)$ (denoted $\mathcal{M}_{\Gamma(N)\cap\Gamma_0(p)}$ in loc. cit.; in \cite{deligne} one only considers the dense open $M_{N,p}=\mathcal{M}_{\Gamma(N)\cap\Gamma_0(p)}^0$) classifying isomorphism classes of $p$-isogenies $\phi:(E\to S,\alpha)\to (E'\to S,\alpha')$. It is smooth away from $p$ and has semi-stable reduction at $p$. Consider the two canonical morphisms
\[ q_1:X(N,p)\to X(N):\phi\mapsto (E\to S,\alpha) \]
and
\[ q_2:X(N,p)\to X(N):\phi\mapsto (E'\to S,\alpha'). \]
One can show that each $q_i$ is finite and flat. The universal object over $X(N,p)$ is a $p$-isogeny
\[ \phi:q_1^*E\to q_2^*E \]
where $E\to X(N)$ is the universal curve. By definition (cf. \cite[3.18]{deligne}), the Hecke correspondence $T_p$ on $E$ is the correspondence
\[ \xymatrix{
& q_1^*E \ar[dl]_{q_2\circ\phi} \ar[dr]^{q_1} & \\
E & & E
} \]
that is, $T_p=q_{1,*}\phi^*q_2^*$. Now the Eichler-Shimura relation states that
\[ T_p|_{E_k}=F+I_p\trans{F} \]
where $F$ is the Frobenius, $\trans{F}$ is its transpose as a correspondence, and $I_p$ is the morphism of $X(N)$ defined $I_p(\mathcal{E},\alpha):=(\mathcal{E},p\alpha)$. This can be proven in the same way as \cite[4.8]{deligne}, and we sketch the proof.

We first check the equality over a dense open of $X(N,p)$. As shown in \cite[4.3]{deligne}, there is a dense open of $X(N,p)_{\mathbb{F}_p}$ isomorphic to the disjoint union of two copies of the dense open $Y(N)_{\mathbb{F}_p}^h\subset Y(N)_{\mathbb{F}_p}$, complement of the supersingular points. Over one of these copies we have $q_1=\text{id}$, and $q_2=F_{Y(N)}$ the Frobenius of $Y(N)^h_{\mathbb{F}_p}$, over the other $q_1=F_{Y(N)}$ and $q_2=I_p$ (cf. \cite[V, 1.17]{delrap}). In the first case the universal isogeny is the relative Frobenius $F_{E/Y(N)}:(E,\alpha)\to (F^*_{Y(N)}E,F^*_{Y(N)}\alpha)$, in the second the universal isogeny is its transpose, the Verschiebung $\trans{F_{E/Y(N)}}:(F^*_{Y(N)}E,F^*_{Y(N)}\alpha)\to (E,p^{-1}\alpha)=I_p^*(E,\alpha)$. So in the first case the Hecke correspondence is
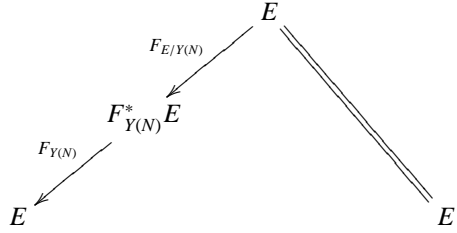
\[ \xymatrix{
& & E \ar[dl]_{F_{E/Y(N)}} \ar@{=}[ddrr] & & \\
& F_{Y(N)}^*E \ar[dl]_{F_{Y(N)}} & & & \\
E & & & & E
} \]
which is obviously $F$. In the second case it is
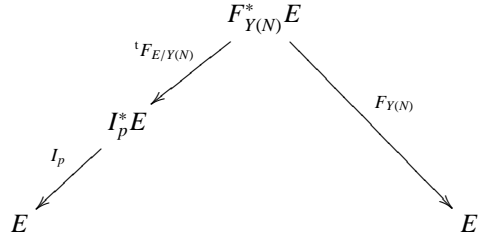
\[ \xymatrix{
& & F^*_{Y(N)}E \ar[dl]_{\trans{F_{E/Y(N)}}} \ar[ddrr]^{F_{Y(N)}} & & \\
& I_p^*E \ar[dl]_{I_p} & & & \\
E & & & & E
} \]
i.e. $I_p\trans{F}$.

For the general case, first recall that the open immersion $Y(N)^h_{\mathbb{F}_p}\coprod Y(N)^h_{\mathbb{F}_p}\to X(N,p)_{\mathbb{F}_p}$ extends to a morphism $X(N)_{\mathbb{F}_p}\coprod X(N)_{\mathbb{F}_p}\to X(N,p)_{\mathbb{F}_p}$ which identifies the domain with the disjoint union of the irreducible components of the target \cite[V, \S1]{delrap}, hence is surjective. By the above this gives a surjective morphism $E_k\coprod F_{X(N)}^*E_k\to q_1^*E_k$. Therefore, $\im(q_1^*E_k\to E_k\times_k E_k)=\im(E_k\coprod F_{X(N)}^*E_k\to E_k\times_kE_k)$ has at most two irreducible components. Its irreducible components are therefore the transpose of the graph of $F$ and the graph of $FI_p^{-1}$ (note that these are irreducible since $E_k$ is). This proves the Eichler-Shimura relation.

Now since $p\equiv 1\mod N$, $I_p$ is the identity map and we obtain the relation
\begin{equation}\label{ES} T_p|_{E_{k}}=F+\trans{F}. \end{equation}
Since we have $p\varphi=F$ and $\trans{F}F=p^2$ ($\dim E_k=2$), it follows that $pT=T_p|_{E_{k}}$.

Finally, let $Z:=\im(q_1^*E\to E\times_WE)$. Note that since $q_1^*E$ is flat over $W$, so is $Z$. The correspondence $T_p$ acts as $\pr_{2,*}\circ c_Z\circ\pr_1^*$, where $c_Z$ is the cup product with the characteristic class of $Z$, and $\pr_i:E\times_WE\to E$ are the projections. Therefore, to complete the proof it remains only to notice that the canonical isomorphism $H^*_{\dR}(E_{K_0}\times_{K_0}E_{K_0})=H^*_{\cris}(E_k\times_kE_k/K_0)$ is compatible with taking characteristic classes of closed subschemes of $E\times_WE$ which are flat over $W$. Indeed, the isomorphism is compatible with Chern classes of line bundles, hence also (by the splitting principle) with Chern classes of vector bundles, and the structure sheaf of a closed subscheme of $E\times_WE$ has a finite resolution by vector bundles which, if flat over $W$, reduces to give a resolution on $E_k\times_kE_k$.
\end{proof}

As a corollary we obtain Theorem \ref{main}.

\begin{corollary}\label{tate}
If $k=\mathbb{F}_p$ and $p\equiv 1\mod N$, then under hypothesis (\ref{SS}) we have
\begin{enumerate}[(i)]
\item Tate's conjecture holds for $E_k$
\item $NS(E_{\bar{K}})^{G_{K_0}}\otimes\mathbb{Q}=NS(E_{\bar{k}})^{G_{K_0}}\otimes\mathbb{Q}$
\item the Mordell-Weil group of the generic fibre of $E_k\to X_k$ is isomorphic to $(\mathbb{Z}/N)^2$.
\end{enumerate}
\end{corollary}
\begin{proof}
Note that (\ref{SS}) implies that $D_{\cris}(V_p\Br(E(N)_{\bar{K}})^{(\varphi-1)^2=0}=D_{\cris}(V_p\Br(E(N)_{\bar{K}}))^{\varphi=1}$. So by Theorem \ref{eichlershimura}, (i) and (ii) follow from Corollary \ref{maincor}. It follows from (ii) that the Mordell-Weil group is finite, hence isomorphic to $(\mathbb{Z}/N)^2$ by \cite[Appendix, Cor. 3]{shioda1}.
\end{proof}

\subsection{Validity of (\ref{SS})}
We can show that (\ref{SS}) holds for all primes outside of a set of density zero. For the sake of brevity we only sketch the argument. Let $Y_1(N)$ denote the Deligne-Mumford moduli stack of pairs $(\mathcal{E}\to S,P,P')$ where $\mathcal{E}\to S$ is an elliptic curve over a $\mathbb{Z}[1/N]$-scheme $S$, $P\in\mathcal{E}[N](S)$ a point of exact order $N$ and $P'\in(\tfrac{\mathcal{E}[N]}{<P>})(S)\cong\mu_N(S)$ a point of exact order $N$. For $N\geq 5$ it is known to be a $\mathbb{Z}[1/N,\zeta_N]$-scheme with geometrically connected fibres. Let $g:E_1(N)\to Y_1(N)$ be the universal elliptic curve and consider $V_N:=\tilde{H}^1(Y_1(N)_{\bar{K}},R^1g_*\mathbb{Q}_p)(1)$. We will first explain why (\ref{SS}) holds for $D_{\cris}(V_N)$ outside a set of primes of density zero.

For every proper divisor $d$ of $N$ there are pairs of maps $\pi_i:V_{N/d}\to V_N$ ($i=1,2$) defined just like for modular forms. (One map is just the canonical ``inclusion'' and the other arises from the inclusion $\bigl(\begin{smallmatrix}
d&0\\ 0&1
\end{smallmatrix} \bigr)\Gamma_1(N)\bigl(\begin{smallmatrix}
d&0\\ 0&1
\end{smallmatrix} \bigr)^{-1}\subset\Gamma_1(N/d)$.) The image of these maps is a subspace $V_N^{\old}$ of $V_N$. Its orthogonal complement under the cup product is a subspace $V_N^{\new}$ which corresponds to newforms. The latter splits under the action of the Hecke algebra as a direct sum
\[ V_N^{\new}=\oplus_{i=1}^mV(f_i) \]
where $f_1,...,f_m$ are a choice of representatives of the $\Gal(\bar{\mathbb{Q}}/\mathbb{Q})$-conjugacy classes of weight 3 normalized newforms for $\Gamma_1(N)$, and $V(f_i)(-1)$ is the Galois representation associated to $f_i$ by Deligne. Now, by autoduality we have
\[ V_N^{\old}\subset\prod_{d|N,d>1}V_{N/d}^{\oplus 2} \]
so by induction on $N$ we may assume (\ref{SS}) to hold for $D_{\cris}(V_N^{\old})$. The space $V(f_i)$ is free of rank 2 over $K_{f_i}\otimes\mathbb{Q}_p$, where $K_{f_i}$ is the field of coefficients of $f_i$. If $\varphi$ is the Frobenius, we want to show that $D_{\cris}(V(f_i))^{(\varphi^n-1)^2=0}=D_{\cris}(V(f_i))^{\varphi^n=1}$ for some fixed $n$. This is immediate if $f_i$ has CM: in this case $\varphi^2$ acts as a diagonal matrix \cite{ribet}. If $f_i$ does not have CM and (\ref{SS}) does not hold, then the minimal polynomial of $\varphi^n$ is $(t-1)^2$, so both eigenvalues of $\varphi^n$ are equal to $1$. Since the trace of $\varphi$ is equal to $\tfrac{a_p}{p}$, where $a_p$ is the $p$th coefficient of $f_i$, this implies that $a_p=(\zeta+\zeta')p$ for some $n$th roots of unity $\zeta,\zeta'$. By \cite{serrecebotarev}, this can only happen for a set of primes of density zero (dependent on the choice of $n$). For given $N$ and $K_0$ we can take $n$ such that $p^n=\#k$.

Finally, if $\zeta_{N^2}\in K_0$, then there is a dominant morphism $Y_1(N^2)_{K_0}\to Y(N)_{K_0}$ arising from the inclusion $\bigl(\begin{smallmatrix}
N&0\\ 0&1
\end{smallmatrix} \bigr)\Gamma_1(N^2)\bigl(\begin{smallmatrix}
N&0\\ 0&1
\end{smallmatrix}\bigr)^{-1}\subset\Gamma(N)$. From this one can deduce (\ref{SS}) for $D_{\cris}(\tilde{H}^1(Y(N)_{\bar{K}},R^1g_*\mathbb{Q}_p(1)))=D_{\cris}(V_p\Br(E(N)_{\bar{K}}))$.

\begin{corollary}\label{tatecor}
For $k=\mathbb{F}_p$ and $p\equiv 1\mod N$, the conclusion of Corollary \ref{tate} holds outside of a set of primes of density zero.
\end{corollary}

\section{The finite height case}
Let $W\subset V$ be a finite extension of discrete valuation rings, $k_V$ (resp. $K$) the residue (resp. fraction) field of $V$, and set $T:=\Spec(V)$.

\subsection{The formal Brauer group}

Let $f:X\to T$ be a proper smooth morphism. Artin and Mazur considered the sheaves
\[ S\rightsquigarrow\widehat{\mathbb{G}}_m(S):=\ker(\mathbb{G}_m(S)\to\mathbb{G}_m(S_{\red})) \]
on the big \'etale site of $X$ and
\[ \Phi^i:=R^if_*\widehat{\mathbb{G}}_m \]
on the big \'etale site of $T$, and proved the following (\cite[II, 2.12]{artinmazur}).

\begin{theorem}[Artin-Mazur]
If $\Phi^{i-1}$ is a formally smooth functor and $H^{i+1}(X_{k_V},\mathcal{O}_{X_{k_V}})=0$, then $\Phi^i$ is pro-representable by a formal Lie group over $V$.
\end{theorem}

\begin{corollary}
Assume $\dim X_{k_V}=2$. If $\Phi^1$ is formally smooth, then $\Phi^2$ is pro-representable by a formal Lie group over $V$, called the formal Brauer group.
\end{corollary}

Now if we assume $\Phi^2$ is a formal Lie group of \emph{finite height} (i.e. a connected $p$-divisible group), then we obtain the following.

\begin{corollary}\label{repcor}
Assume $k$ is finite and $\dim X_{k_V}=2$. Let $\bar{V}$ be the normalization of $V$ in $\bar{K}$. If $\Phi^1$ is formally smooth and $\Phi^2$ is of finite height, then the canonical map
\[ \rho:\varinjlim_{V'\subset\bar{V}}\varprojlim_nH^2(X_{V'/p^nV'},\mathbb{G}_m)\to\Br(X_{\bar{k}}) \]
is surjective on $p$-primary torsion components, where the direct limit is taken over all finite extensions of discrete valuation rings $V\subset V'$ contained in $\bar{V}$.
\end{corollary}
\begin{proof}
It suffices to show that $\rho$ is surjective and $\ker(\rho)$ is a $p$-divisible abelian group. To this end, consider the exact sequence of \'etale sheaves on $X_{V'/p^nV'}$
\begin{equation}\label{gmdevissage} 1\to\widehat{\mathbb{G}}_m\to\mathbb{G}_{m,X_{V'/p^nV'}}\to\mathbb{G}_{m,X_{k_{V'}}}\to 1.
\end{equation}
Note that since $V'$ is a henselian discrete valuation ring with finite residue field for any torsion \'etale sheaf $\mathcal{F}$ on $V'$ we have $H^i(V',\mathcal{F})=0$ for $i>1$ and $H^1(V',\mathcal{F})=(\mathcal{F}(\bar{k}))_{G_{k_{V'}}}$ (coinvariants), where $G_{k_{V'}}:=\Gal(\bar{k}/k_{V'})$. Let $V'\subset\mathcal{V}$ be the maximal unramified extension. From the Leray spectral sequence for the Galois covering $V'\subset\mathcal{V}$ we get an exact sequence
\[ 0\to H^1(X_{\mathcal{V}/p^n\mathcal{V}},\widehat{\mathbb{G}}_m)_{G_{k_{V'}}}\to H^2(X_{V'/p^nV'},\widehat{\mathbb{G}}_m)\to H^2(X_{\mathcal{V}/p^n\mathcal{V}},\widehat{\mathbb{G}}_m)^{G_{k_{V'}}}\to 0 \]
i.e.
\[ 0\to\Phi^1(\mathcal{V}/p^n\mathcal{V})_{G_{k_{V'}}}\to H^2(X_{V'/p^nV'},\widehat{\mathbb{G}}_m)\to \Phi^2(V'/p^nV')\to 0. \]
Now since the ideal $\ker(\mathcal{O}_{X_{V'/p^nV'}}\to\mathcal{O}_{X_{k_{V'}}})$ is nilpotent, the sheaf $\widehat{\mathbb{G}}_m$ has a finite filtration with quotients isomorphic to $\mathcal{O}_{X_{k_{V'}}}$, and this implies that $H^i(X_{V'/p^nV'},\widehat{\mathbb{G}}_m)$ is finite for all $i$. So taking the inverse limit in the exact sequence we obtain an exact sequence
\[ 0\to\varprojlim_n\Phi^1(\mathcal{V}/p^n\mathcal{V})_{G_{k_{V'}}}\to\varprojlim_nH^2(X_{V'/p^nV'},\widehat{\mathbb{G}}_m)\to\Phi^2(V')\to 0. \]
Now assume the following:
\begin{equation}\label{assump}\tag{D} \varinjlim_{V'\subset\mathcal{V}}\varprojlim_n\Phi^1(\mathcal{V}/p^n\mathcal{V})_{G_{k_{V'}}}\;\text{is a $p$-divisible abelian group.}
\end{equation}
This implies that
\[ \varinjlim_{V'\subset\bar{V}}\varprojlim_nH^2(X_{V'/p^nV'},\widehat{\mathbb{G}}_m) \]
is a $p$-divisible abelian group: by the above we have a surjective map
\[ \varinjlim_{V'\subset\bar{V}}\varprojlim_nH^2(X_{V'/p^nV'},\widehat{\mathbb{G}}_m)\to\varinjlim_{V'\subset\bar{V}}\Phi^2(V')=\Phi^2(\bar{V}) \]
whose kernel is $p$-divisible by (\ref{assump}), and by \cite[Cor. to Prop. 4]{tate} $\Phi^2(\bar{V})$ is a $p$-divisible abelian group. Taking cohomology in the exact sequence (\ref{gmdevissage}) we get an exact sequence
\[ H^2(X_{V'/p^nV'},\widehat{\mathbb{G}}_m)\to H^2(X_{V'/p^nV'},\mathbb{G}_m)\to\Br(X_{k_{V'}})\to 0 \]
(the exactness on the right follows from $\dim X_{k_{V'}}=2$). Since $H^2(X_{V'/p^nV'},\widehat{\mathbb{G}}_m)$ is finite taking limits we get an exact sequence
\[ \varinjlim_{V'\subset\bar{V}}\varprojlim_nH^2(X_{V'/p^nV'},\widehat{\mathbb{G}}_m)\to\varinjlim_{V'\subset\bar{V}}\varprojlim_nH^2(X_{V'/p^nV'},\mathbb{G}_m)\overset{\rho}{\to}\Br(X_{\bar{k}})\to 0 \]
and this implies the result.

So it suffices to show (\ref{assump}). We claim that $M:=\left(\varprojlim_n\Phi^1(\mathcal{V}/p^n\mathcal{V})\right)\otimes\mathbb{Z}/p$ is a discrete $G_{k_{V'}}$-module (i.e. the stabilizer of every element is open in $G_{k_{V'}}$). Indeed, taking cohomology in (\ref{gmdevissage}) we get an exact sequence
\[ 0\to \Phi^1(\mathcal{V}/p^n\mathcal{V})\to\Pic(X_{\mathcal{V}/p^n\mathcal{V}})\to\Pic(X_{k_{\mathcal{V}}}) \]
whence an exact sequence
\[ 0\to\varprojlim_n\Phi^1(\mathcal{V}/p^n\mathcal{V})\to\varprojlim_n\Pic(X_{\mathcal{V}/p^n\mathcal{V}})\to\Pic(X_{k_{\mathcal{V}}}). \]
Now, denoting by $\widehat{X_{\mathcal{V}}}$ the formal completion of $X_{\mathcal{V}}$ along its special fibre over $\Spec(\mathcal{V})$, we have
\[ \varprojlim_n\Pic(X_{\mathcal{V}/p^n\mathcal{V}})=\Pic(\widehat{X_{\mathcal{V}}}) \]
so it clearly suffices to show that $\Pic(\widehat{X_{\mathcal{V}}})\otimes\mathbb{Z}/p$ is a discrete $G_{k_{V'}}$-module. But taking cohomology in the exact sequence on $\widehat{X_{\mathcal{V}}}$
\[ 1\to 1+p^2\mathcal{O}_{\widehat{X_{\mathcal{V}}}}\to\mathbb{G}_{m,\widehat{X_{\mathcal{V}}}}\to\mathbb{G}_{m,X_{\mathcal{V}/p^2\mathcal{V}}}\to 1 \]
we get an exact sequence
\[ H^1(\widehat{X_{\mathcal{V}}},1+p^2\mathcal{O}_{\widehat{X_{\mathcal{V}}}})\to\Pic(\widehat{X_{\mathcal{V}}})\to\Pic(X_{\mathcal{V}/p^2\mathcal{V}}) \]
and one sees easily (cf. \S\ref{comparisonsection}) that the (usual) logarithm induces an isomorphism
\[ \log:H^1(\widehat{X_{\mathcal{V}}},1+p^2\mathcal{O}_{\widehat{X_{\mathcal{V}}}})\cong H^1(\widehat{X_{\mathcal{V}}},p^2\mathcal{O}_{\widehat{X_{\mathcal{V}}}}) \]
and the claim follows from this. Finally, note that since $\Phi^1$ is formally smooth by assumption, the maps $\Phi^1(\mathcal{V}/p^{n+1}\mathcal{V})\to\Phi^1(\mathcal{V}/p^n\mathcal{V})$ are surjective for all $n$, so in particular $\varprojlim_n^{(1)}\Phi^1(\mathcal{V}/p^n\mathcal{V})=0$. If $\gamma\in G_{k_{V'}}$ is a topological generator, then we have an exact sequence
\[ 0\to \Phi^1(\mathcal{V}/p^n\mathcal{V})^{G_{k_{V'}}}\to \Phi^1(\mathcal{V}/p^n\mathcal{V})\overset{\gamma-1}{\to}\Phi^1(\mathcal{V}/p^n\mathcal{V})\to \Phi^1(\mathcal{V}/p^n\mathcal{V})_{G_{k_{V'}}}\to 0. \]
Noting that $\Phi^1(\mathcal{V}/p^n\mathcal{V})^{G_{k_{V'}}}=\Phi^1(V'/p^nV')$, hence $\varprojlim_n^{(1)}\Phi^1(\mathcal{V}/p^n\mathcal{V})^{G_{k_{V'}}}=0$ (again by the formal smoothness of $\Phi^1$), and taking limits in this sequence we find
\[ \varprojlim_n\Phi^1(\mathcal{V}/p^n\mathcal{V})_{G_{k_{V'}}}=\left(\varprojlim_n\Phi^1(\mathcal{V}/p^n\mathcal{V})\right)_{G_{k_{V'}}}. \]
Since direct limits commute with tensor product we get
\[ \left(\varinjlim_{V'\subset\mathcal{V}}\varprojlim_n\Phi^1(\mathcal{V}/p^n\mathcal{V})_{G_{k_{V'}}}\right)\otimes\mathbb{Z}/p=\varinjlim_{V'\subset\mathcal{V}}\left((\varprojlim_n\Phi^1(\mathcal{V}/p^n\mathcal{V}))\otimes\mathbb{Z}/p\right)_{G_{k_{V'}}}=\varinjlim_{V'\subset\mathcal{V}}M_{G_{k_{V'}}}. \]
Now as $M$ is a torsion discrete $G_{k_{V'}}$-module we have
\[ \varinjlim_{V'\subset\mathcal{V}}M_{G_{k_{V'}}}=\varinjlim_{V'\subset\mathcal{V}}H^1(k_{V'},M)=0 \]
which proves (\ref{assump}).
\end{proof}

\subsubsection{Example: elliptic surfaces}\label{elliptic}
Suppose that $f:X\to T$ is a projective flat morphism with geometrically integral fibres. Then $R^1f_*\mathbb{G}_m$ is representable by a scheme $\Pic_{X/T}$ which is locally of finite type over $T$ (\cite[9.4.8]{kleiman}). Moreover, for all $t\in T$, $H^1(X_t,\mathcal{O}_{X_t})$ is the tangent space of $(\Pic_{X/T})_t=\Pic_{X_t/t}$ at the origin (\cite[9.5.11]{kleiman}), hence $\Pic_{X_t/t}$ is smooth over $t$ if and only if
\begin{equation*} \dim\Pic_{X_t/t}\geq\dim H^1(X_t,\mathcal{O}_{X_t})
\end{equation*}
in which case we have equality since the reverse inequality always holds.

Now let $Y\to T$ be a smooth projective curve, and let $g:X\to Y$ be an elliptic fibration, i.e. $g$ is a projective morphism with a section whose geometric fibres are connected of dimension 1 and whose generic fibre in each fibre over $T$ is smooth of genus 1. Assume further that $X_t$ is a non-isotrivial smooth elliptic surface for all $t\in T$ (in particular $X$ is smooth over $T$). Then the map
\[ \Pic^0_{Y_t/t}\to\Pic^0_{X_t/t} \]
is an isomorphism for all $t\in T$ (\cite[th. 6.12]{schuettshioda}). In particular, the map $H^1(Y_t,\mathcal{O}_{Y_t})\to H^1(X_t,\mathcal{O}_{X_t})$ is an isomorphism for all $t\in T$, hence $\dim H^1(X_t,\mathcal{O}_{X_t})$ is independent of $t$, since this is true for $Y$. $\Pic_{X_K/K}$ being smooth ($\cha(K)=0$), we get $\dim\Pic_{X_K/K}=\dim H^1(X_k,\mathcal{O}_{X_k})$. Since the connected component $\Pic^0_{X/T}$ of $0$ in $\Pic_{X/T}$ is proper over $T$ (\cite[9.6.18]{kleiman}), we have
\[ \dim\Pic_{X_{k_V}/k_V}=\dim\Pic^0_{X_{k_V}/k_V}\geq\dim\Pic^0_{X_K/K}=\dim H^1(X_{k_V},\mathcal{O}_{X_{k_V}}) \]
so $\Pic_{X_{k_V}/k_V}$ is smooth. By \cite[9.5.20]{kleiman}, this implies that $\Pic^0_{X/T}$ is smooth over $T$. (Note however that this does not imply that $\Pic_{X/T}$ is smooth over $T$, cf. \cite[9.5.22]{kleiman}.) Since $\Phi^1$ is the completion of $\Pic^0_{X/T}$ along its zero section, it is formally smooth in this case.

\subsection{Relationship between $\Br(X_{\bar{V}})$ and $\Br(X_{\bar{k}})$}\label{comparisonsection}
For a scheme $S$ write $S_n:=S\otimes\mathbb{Z}/p^n$. Let $f:X\to T$ be a proper smooth morphism with geometrically connected fibres, and let $\hat{X}:=\varinjlim_nX_n$ be the $p$-adic completion of $X$. In the sequel let $j:X_K\subset X$ and $i:X_1\subset X$ be the inclusions. Then we have exact sequences of \'etale sheaves on $X$
\[ 0\to j_!\mathbb{G}_m\to\mathbb{G}_m\to i_*i^*\mathbb{G}_m\to 0 \]
and
\[ 0\to 1+p^2i_*i^*\mathcal{O}_{X}\to i_*i^*\mathbb{G}_m\to i_*\mathbb{G}_{m,X_2}\to 0 \]
where, for a morphism $g$ of schemes, $g^*$ always denotes the inverse image functor and not module pullback.

\begin{lemma}\label{torsion}
We have $i^*\mathbb{G}_m[p^n]=i^*f^*(\mathcal{O}^*_T[p^n])=\mathbb{G}_{m,\hat{X}}[p^n]$ for all $n$, where $\mathbb{G}_{m,\hat{X}}=\varprojlim_n\mathbb{G}_{m,X_n}$.
\end{lemma}
\begin{proof}
Since the generic fibre of $f$ is geometrically connected, $K$ is algebraically closed in the fraction field of $X$. Let $x\to X_{k_V}$ be a geometric point and let $V'$ be the normalization of $V$ in $\mathcal{O}_{X,x}$. If $\pi\in V$ is a uniformizer, by flatness we have $V'/\pi V'\subset\mathcal{O}_{X,x}/\pi\mathcal{O}_{X,x}=\mathcal{O}_{X_{k_V},x}$, hence $\pi\in V'$ is also a uniformizer and $V'$ is unramified over $V$. In particular any $p$-power root of unity in $\mathcal{O}_{X,x}$ already lies in $V$, i.e. $i^*\mathbb{G}_m[p^n]=i^*f^*(\mathcal{O}^*_T[p^n])$. For the second equality, note that since $X$ is excellent, the formal fibres of $\mathcal{O}_{X,x}$ are geometrically integral (\cite[IV, 18.9.1]{ega}), hence $\mathcal{O}_{X,x}$ is integrally closed in its completion (since $X$ is normal). In particular, $\mathbb{G}_{m,\hat{X}}[p^n]=i^*\mathbb{G}_m[p^n]$.
\end{proof}

\begin{lemma}\label{logexp}
Let $A$ be a $p$-adically complete ring. For all $n\geq 1$ the logarithm induces a natural isomorphism $1+p^{n+1}A\cong p^{n+1}A$. In particular, $(1+p^{n+1}A)^p=1+p^{n+2}A$.
\end{lemma}
\begin{proof}
The logarithm induces a map
\begin{eqnarray*}
\log:1+p^{n+1}A &\to& p^{n+1}A \\
1+p^{n+1}x &\mapsto& -p^{n+1}\sum_{r>0}\dfrac{p^{(n+1)(r-1)}(-x)^{r}}{r}
\end{eqnarray*}
and the exponential
\begin{eqnarray*}
\exp:p^{n+1}A &\to& 1+p^{n+1}A \\
p^{n+1}x &\mapsto& 1+p^{n+1}\sum_{r>0}\dfrac{p^{(n+1)(r-1)}x^{r}}{r!}
\end{eqnarray*}
is its inverse. Recall that $v_p(r!)\leq\dfrac{r-1}{p-1}$, so in particular $v_p(r!)<(n+1)(r-1)$ and $\lim_{r\to+\infty}(n+1)(r-1)-v_p(r!)=+\infty$ for all $n\geq 1$, hence these homomorphisms are well-defined.
\end{proof}

\begin{proposition}\label{logexp2}
The inclusion $(1+p^{n+1}i_*i^*\mathcal{O}_{X})^p\subset 1+p^{n+2}i_*i^*\mathcal{O}_{X}$ is bijective and $\left(1+p^{n+1}i_*i^*\mathcal{O}_{X}\right)\otimes\mathbb{Z}/p^m\cong p^{n+1}\mathcal{O}_X/p^{n+1+m}\mathcal{O}_X$ for all $n,m\geq 1$.
\end{proposition}
\begin{proof}
This is local so we may assume that $X=\Spec(R)$ is the spectrum of a henselian regular local ring with $p$ in its maximal ideal. If $\hat{R}$ denotes the $p$-adic completion we have a commutative diagram with exact rows
\[ \xymatrix{
1 \ar [r] & (1+p^{n+1}R)^p \ar[d] \ar[r] & 1+p^{n+1}R \ar[d] \ar[r] & C \ar[d] \ar[r] & 0 \\
1 \ar[r] & (1+p^{n+1}\hat{R})^p \ar[r] & 1+p^{n+1}\hat{R}  \ar[r] & p^{n+1}R/p^{n+2}R  \ar[r] & 0
} \]
where the middle vertical map is injective and the bottom right term follows from Lemma \ref{logexp}. The statement amounts to showing that the right vertical map is injective since it is already clearly surjective. For this it suffices to show that the cokernel $D$ of the left vertical map is $p$-torsion free. This follows immediately from the fact, already noted in the proof of Lemma \ref{torsion}, that $R$ is integrally closed in $\hat{R}$.
\end{proof}

\begin{corollary}\label{logexp3}
For all $n$ the canonical map
\[ \mathbb{G}_{m,X}/p^{n}\to i_*\mathbb{G}_{m,\hat{X}}/p^{n} \]
is bijective.
\end{corollary}
\begin{proof}
Since $j_!\mathbb{G}_m$ is $p$-divisible ($j_!$ being exact), we have $\mathbb{G}_{m,X}/p^n=i_*i^*\mathbb{G}_{m,X}/p^n$. Now we have a commutative diagram
\[ \begin{CD}
0 @>>> 1+p^{2}i^*\mathcal{O}_X @>>> i^*\mathbb{G}_{m,X} @>>> \mathbb{G}_{m,X_2} @>>> 0 \\
@. @VVV @VVV @| @. \\
0 @>>> 1+p^{2}\mathcal{O}_{\hat{X}} @>>> \mathbb{G}_{m,\hat{X}} @>>> \mathbb{G}_{m,X_2} @>>> 0 \\
\end{CD} \]
so reducing modulo $p^n$ we see that the result follows from Proposition \ref{logexp2} and the five lemma.
\end{proof}

\begin{proposition}\label{surj1}
The canonical map $\Br(X)\to H^2(X_1,\mathbb{G}_{m,\hat{X}})$ is surjective on $p$-primary torsion components.
\end{proposition}
\begin{proof}
Let $\mathcal{T}\subset\mathbb{G}_{m,X}$ and $\mathcal{F}\subset\mathbb{G}_{m,\hat{X}}$ be the $p$-power torsion subsheaves, and define $\mathbb{G}'_{m,X}:=\mathbb{G}_m/\mathcal{T}$ and $\mathbb{G}'_{m,\hat{X}}:=\mathbb{G}_m/\mathcal{F}$. By Lemma \ref{torsion} $\mathcal{F}$ is a finite constant sheaf and we have an exact sequence
\begin{equation}\label{exseqtorsion} 0\to j_!\mathbb{Q}_p/\mathbb{Z}_p(1)\to\mathcal{T}\to\mathcal{F}\to 0.
\end{equation}
By the proper base change theorem we have $H^n(X,\mathcal{T})=H^n(X_1,\mathcal{F})=H^n(X,\mathcal{F})$ for all $n$, hence we deduce a commutative diagram with exact rows
\[ \begin{CD}
H^2(X,\mathcal{F}) @>>> H^2(X,\mathbb{G}_m) @>>> H^2(X,\mathbb{G}'_{m,X}) @>>> H^3(X,\mathcal{F}) \\
@| @VVV @VVV @| \\
H^2(X_1,\mathcal{F}) @>>> H^2(X_1,\mathbb{G}_{m,\hat{X}}) @>>> H^2(X_1,\mathbb{G}'_{m,\hat{X}}) @>>> H^3(X_1,\mathcal{F})
\end{CD} \]
and since $\Br(X)=H^2(X,\mathbb{G}_m)$ is torsion (\cite[II, 1.4]{brauer}) an easy diagram chase shows that it suffices to prove that the map
\[ H^2(X,\mathbb{G}'_{m,X})\to H^2(X_1,\mathbb{G}'_{m,\hat{X}}) \]
is surjective on $p$-primary torsion. Now since the functor $j_!$ is exact, $j_!\mathbb{Q}_p/\mathbb{Z}_p(1)$ is divisible, so from the sequence (\ref{exseqtorsion}) for all $n$ we get an exact sequence
\[ 0\to\mathcal{F}/p^n\to\mathbb{G}_{m,X}/p^n\to\left(\mathbb{G}'_{m,X}\right)/p^n\to 0 \]
and therefore by Corollary \ref{logexp3} we deduce isomorphisms
\[ \mathbb{G}_{m,X}'/p^n=\mathbb{G}_{m,\hat{X}}'/p^n \]
for all $n$. Hence for all $n$ we have a commutative diagram with exact rows
\[ \begin{CD}
H^1(X,\mathbb{G}'_{m,X}/p^n) @>>> H^2(X,\mathbb{G}'_{m,X})[p^n] @>>> 0 \\
@| @VVV @. \\
H^1(X_1,\mathbb{G}_{m,\hat{X}}'/p^n) @>>> H^2(X_1,\mathbb{G}_{m,\hat{X}}')[p^n] @>>> 0
\end{CD} \]
and the result follows.
\end{proof}

Denote by $\varprojlim_n^{(m)}$ the right derived functors of $\varprojlim_n$.

\begin{proposition}\label{limexact}
\mbox{}
\begin{enumerate}[(i)]
\item Suppose $\{\mathcal{F}_n\}_{n\in\mathbb{N}}$ is an inverse system of sheaves of abelian groups on $X_1$ such that for any \'etale map $h:U\to X_1$ with $U$ affine we have
\begin{enumerate}[(a)]
\item the transition maps $\mathcal{F}_{n+1}(U)\to\mathcal{F}_n(U)$ are surjective for all $n$
\item $H^i(U,\mathcal{F}_n)=0$ for all $i>0$ and any $n$.
\end{enumerate}
Then $\varprojlim_n^{(m)}\mathcal{F}_n=0$ for all $m\geq 1$.
\item $\varprojlim_n^{(m)}\mathbb{G}_{m,X_n}=0$ for $m\geq 1$.
\item There are canonical isomorphisms $H^i(X_1,\mathbb{G}_{m,\hat{X}})=H^i_{\cont}(X_1,\{\mathbb{G}_{m,X_n}\})$ for all $i$, where $H^*_{\cont}$ is the continuous \'etale cohomology, i.e. the derived functor of $\varprojlim_nH^0(X_1,-)$ (\cite{jannsen}).
\end{enumerate}
\end{proposition}
\begin{proof}
First note that since $\Hom$ commutes with products in the right argument, $h^*$ commutes with products. If $d_n:\mathcal{F}_{n+1}\to\mathcal{F}_n$ are the transition maps, then we have a left-exact sequence
\[ 0\to\varprojlim_n\mathcal{F}_n\to\prod_n\mathcal{F}_n\overset{\text{id}-d_n}{\longrightarrow}\prod_n\mathcal{F}_n\to 0 \]
which is also right-exact by virtue of property (a) and the well-known exact sequence (\cite[1.4]{jannsen})
\[ 0\to\varprojlim_n\mathcal{F}_n(U)\to\prod_n\mathcal{F}_n(U)\overset{\text{id}-d_n}{\longrightarrow}\prod_n\mathcal{F}_n(U)\to\varprojlim_n{}^{(1)}\left(\mathcal{F}_n(U)\right)\to 0. \]
An object $\{\mathcal{I}_n\}_n$ of the category of inverse systems of sheaves of abelian groups is injective if and only if each $\mathcal{I}_n$ is injective and the transition maps $d_n:\mathcal{I}_{n+1}\to \mathcal{I}_{n}$ are split surjective (\cite[1.1]{jannsen}). In particular, the sequence
\[ 0\to\varprojlim_n\mathcal{I}_n\to\prod_n\mathcal{I}_n\overset{\text{id}-d_n}{\longrightarrow}\prod_n\mathcal{I}_n\to 0 \]
is exact. Now let $\{\mathcal{F}_n\}\subset\{\mathcal{I}_n\}$ with $\{\mathcal{I}_n\}$ an injective inverse system of sheaves of abelian groups, and consider the exact sequence
\[ 0\to \mathcal{F}_n\to \mathcal{I}_n\to \mathcal{G}_n\to 0 \]
where $\mathcal{G}_n=\mathcal{I}_n/\mathcal{F}_n$. Since for any \'etale $U\to X_1$ with $U$ affine we have $H^1(U,\mathcal{F}_n)=0$, it follows that $\mathcal{G}_n(U)=\mathcal{I}_n(U)/\mathcal{F}_n(U)$. In particular, the sequence
\[ 0\to\prod_n\mathcal{F}_n\to\prod_n\mathcal{I}_n\to\prod_n\mathcal{G}_n\to 0 \]
is exact. From this and the above we deduce an exact sequence
\[ 0\to\varprojlim_n\mathcal{F}_n\to\varprojlim_n\mathcal{I}_n\to\varprojlim_n\mathcal{G}_n\to 0 \]
hence $\varprojlim_n^{(1)}\mathcal{F}_n=0$ and $\varprojlim_n^{(m)}\mathcal{F}_n=\varprojlim_n^{(m-1)}\mathcal{G}_n$ for $m>1$. Since $\{\mathcal{G}_n\}$ also satisfies (a) and (b), part (i) follows by induction on $m$.

For (ii), consider the exact sequence ($n\geq 2$)
\[ 1\to 1+p^2\mathcal{O}_{X_n}\to\mathbb{G}_{m,X_n}\to\mathbb{G}_{m,X_2}\to 1. \]
Since $X$ is flat over $T$, the isomorphism of Lemma \ref{logexp} induces canonical isomorphisms $1+p^2\mathcal{O}_{X_n}\cong p^2\mathcal{O}_{X_n}$, compatible for varying $n$. So taking limits in the exact sequence we get surjective maps
\[ \varprojlim_n{}^{(m)}p^2\mathcal{O}_{X_n}\twoheadrightarrow\varprojlim_n{}^{(m)}\mathbb{G}_{m,X_n} \]
for all $m\geq 1$. Thus it suffices to show that $\varprojlim_n^{(m)}p^2\mathcal{O}_{X_n}=0$ for $m\geq 1$, which follows from (i). Since $\varprojlim_n$ commutes with taking global sections and preserves injectives (\cite[1.17]{jannsen}), (iii) follows from (ii).
\end{proof}

\begin{proposition}
If $k$ is finite, then the canonical map $H^2(X_1,\mathbb{G}_{m,\hat{X}})\to\varprojlim_nH^2(X_1,\mathbb{G}_{m,X_n})=\varprojlim_nH^2(X_n,\mathbb{G}_m)$ is surjective on $p$-primary torsion components.
\end{proposition}
\begin{proof}
By Proposition \ref{limexact} (iii) and \cite[3.1]{jannsen} we have an exact sequence
\[ 0\to\varprojlim_n{}^{(1)}H^1(X_n,\mathbb{G}_{m,X_n})\to H^2(X_1,\mathbb{G}_{m,\hat{X}})\to\varprojlim_nH^2(X_n,\mathbb{G}_{m,X_n}) \to 0 \]
so it suffices to show that $\varprojlim_n^{(1)}H^1(X_n,\mathbb{G}_{m,X_n})=\varprojlim_n^{(1)}\Pic(X_n)$ is $p$-divisible. Now let $P:=\Pic(X_1)$, $P_n:=\text{im}(\Pic(X_n)\to P)$ and $Q_n:=\ker(\Pic(X_n)\to P)$. We have an exact sequence
\[ H^1(X_{k_V},1+p\mathcal{O}_{X_n})\to\Pic(X_n)\to P\to H^2(X_{k_V},1+p\mathcal{O}_{X_n}). \]
Consider the finite filtration of $1+p\mathcal{O}_{X_n}$
\[ 1=1+p^n\mathcal{O}_{X_n}\subset 1+p^{n-1}\mathcal{O}_{X_n}\subset\cdots\subset 1+p\mathcal{O}_{X_n}. \]
It has length $n-2$ and its graded is annihilated by $p$. It follows that for all $i$ the group $H^i(X_{k_V},1+p\mathcal{O}_{X_n})$ is annihilated by $p^{n-1}$. Since $k$ is finite, this group is finite and we have $\varprojlim_n^{(1)}Q_n=0$, $\varprojlim_n^{(1)}\Pic(X_n)=\varprojlim_n^{(1)}P_n$, and for all $n$ we have $p^{n-1}P\subset P_n$, so we have surjective maps $P/p^{n-1}P\to P/P_n$. Since $P$ is a finitely generated abelian group (again since $k$ is finite), $P/p^{n-1}P$ is finite for all $n$ and we get a surjective map
\[ P\otimes\mathbb{Z}_p=\varprojlim_nP/p^{n-1}P\twoheadrightarrow\varprojlim_nP/P_n. \]
So $\varprojlim_n^{(1)}P_n=\left(\varprojlim_nP/P_n\right)/P$ is a quotient of the $p$-divisible abelian group $(P\otimes\mathbb{Z}_p)/P$, hence is also $p$-divisible. 
\end{proof}

Together with Proposition \ref{surj1} this implies

\begin{corollary}\label{surj2}
If $k$ is finite, then the canonical map $\Br(X)\to\varprojlim_nH^2(X_n,\mathbb{G}_m)$ is surjective on $p$-primary torsion components.
\end{corollary}

From Corollary \ref{repcor} we can now deduce the following (cf. \cite[IV, 2.1]{artinmazur}).

\begin{theorem}\label{finiteheight}
Assume $k$ is finite. If $\Phi^1$ is formally smooth and $\Phi^2$ is of finite height, then the canonical map $V_p\Br(X_{\bar{V}})\to V_p\Br(X_{\bar{k}})$ is surjective.
\end{theorem}
\begin{proof}
By Corollary \ref{surj2}, the canonical map
\[ \Br(X_{V'})\to\varprojlim_nH^2(X_{V',n},\mathbb{G}_m) \]
is surjective on $p$-primary torsion, hence so is the map
\[ \Br(X_{\bar{V}})\to\varinjlim_{V'\subset\bar{V}}\varprojlim_nH^2(X_{V',n},\mathbb{G}_m). \]
So by Corollary \ref{repcor} the composition
\[ \Br(X_{\bar{V}})\to\varinjlim_{V'\subset\bar{V}}\varprojlim_nH^2(X_{V',n},\mathbb{G}_m)\to\Br(X_{\bar{k}}) \]
is surjective on $p$-primary torsion. Finally we have $\Br(X_{V'})\subset \Br(X_{V'}\otimes_VK)$ (\cite[II, 1.10]{brauer}), hence $\Br(X_{\bar{V}})\otimes\mathbb{Z}_p\subset\Br(X_{\bar{K}})\otimes\mathbb{Z}_p$ is a finite direct sum of subgroups of $\mathbb{Q}_p/\mathbb{Z}_p$, and this implies that the map
\[ V_p\Br(X_{\bar{V}})\to V_p\Br(X_{\bar{k}}) \]
is surjective.
\end{proof}

\begin{corollary}\label{subquot}
Assume $k$ is finite. If $\Phi^1$ is formally smooth and $\Phi^2$ is of finite height, then $V_p\Br(X_{\bar{k}})$ is a subquotient of $V_p\Br(X_{\bar{K}})$.
\end{corollary}

Applied to our elliptic modular surface $E=E(N)\otimes_{\mathbb{Z}[\zeta_N]}W$ with $N\geq 3$, this implies

\begin{corollary}\label{fhcor}
If $k$ is finite and the formal Brauer group of $E_k$ has finite height, then $T_p\Br(E_{\bar{k}})=0$.
\end{corollary}
\begin{proof}
As noted in \ref{elliptic}, $\Phi^1$ is formally smooth over $W$, so $\Phi^2$ is representable by a formal Lie group over $W$. Moreover, it is easy to see that $\Phi^2$ has finite height since $\Phi^2\otimes_Wk$ does by assumption. So by the last corollary $V_p\Br(E_{\bar{k}})$ is a subquotient of $V_p\Br(E_{\bar{K}})$. It follows from Corollary \ref{HT} that $V_p\Br(E_{\bar{k}})$ is a Hodge-Tate representation whose weights are contained in $\{\pm 1\}$. Since it is an unramified representation, this is only possible if it is zero.
\end{proof}

\bibliographystyle{amsplain}
\bibliography{lmodsurfbib}
\end{document}